\documentclass[11pt]{article}
\usepackage{amssymb}
\usepackage{amsfonts}
\usepackage{amsmath}
\usepackage{mathrsfs}
\usepackage{graphicx}
\usepackage{amsbsy}
\usepackage{theorem}
\usepackage{color}
\usepackage{hyperref}
\usepackage{tikz}
\usepackage[normalem]{ulem}
\usepackage{bbm}
\newtheorem{theorem}{Theorem}[section]
\newtheorem{lemma}[theorem]{Lemma}

\newenvironment{proof}{\noindent{\bf Proof.~}}
{{\mbox{}\hfill {\small \fbox{}}\\}}
\def\NN{\mathbb{N}}

\def\RR{\mathbb{R}}

\def\ZZ{\mathbb{Z}}

\def\inttriple {\int\mbox{\hspace{-3mm}}\int\mbox{\hspace{-3mm}}\int}
\def\intquad {\int\mbox{\hspace{-3mm}}\int\mbox{\hspace{-3mm}}\int\mbox{\hspace{-3mm}}\int}

\def\ds{\displaystyle}

\def\ms{\medskip}
\def\ss{\smallskip}

\def\eps{\varepsilon}
\def\bar#1{{\overline #1}}

\def\pa{\partial}
\def\ud{\ \mathrm d}

\def\calP{{\mathcal P}}

\def\sgn{\ \mathrm{sgn}}

\begin{document}

\title{Relaxation limit of the aggregation equation with pointy potential}

\author{Beno\^it Fabr\`eges\footnote{Univ Lyon, Université Claude Bernard Lyon 1, CNRS UMR 5208, Institut Camille Jordan, 43 blvd. du 11 novembre 1918, F-69622 Villeurbanne cedex, France.} , Fr\'ed\'eric Lagouti\`ere$^{*}$, S\'ebastien Tran Tien$^{*}$, and Nicolas Vauchelet\footnote{Laboratoire Analyse, G\'eom\'etrie et Applications CNRS UMR 7539, Universit\'e Sorbonne Paris Nord, Villetaneuse, France.}
}

\maketitle

\begin{abstract}
  This work is devoted to the study of a relaxation limit of the so-called aggregation equation with a pointy potential in one dimensional space.
  The aggregation equation is by now widely used to model the dynamics of a density of individuals attracting each other through a potential. When this potential is pointy, solutions are known to blow up in final time. For this reason, measure-valued solutions have been defined.
  In this paper, we investigate an approximation of such measure-valued solutions thanks to a relaxation limit in the spirit of Jin and Xin.
  We study the convergence of this approximation and give a rigorous estimate of the speed of convergence in one dimension with the Newtonian potential. We also investigate the numerical discretization of this relaxation limit by uniformly accurate schemes.
\end{abstract}


\section{Introduction}

The so-called aggregation equation has been widely used to model the dynamics of a population of individuals in interaction.
Let $W: \mathbb{R} \rightarrow \mathbb{R}$, sufficiently smooth, be the interaction potential governing the population.
Then, in one dimension in space, the dynamics of the density of individuals, denoted by $\rho$, is governed by the following equation, for $t>0$ and $x\in\RR$,
\begin{equation}\label{eq:aggreg}
  \pa_t \rho + \pa_x(a[\rho] \rho) = 0, \qquad \text{with }\quad a[\rho] = - W'*\rho.
\end{equation}
Such equations appear in many applications in population dynamics: For instance to describe the collective migration of cells by swarming, the motion of bacteria by chemotaxis, the crowd motion, the flocking of birds, or fishes school see e.g. \cite{Morale, Burger, BurgerMorale,TopazBertozzi1,TopazBertozzi2,DolakSchmeiser,JamesVauchelet}.
From a mathematical point of view, these equations have been widely studied. When the potential $W$ is not smooth enough, it is known that weak solutions may blow up in finite time \cite{Andreaaaa,Andreaaaaaa}.
Thus, the existence of weak (measure) solutions has been investigated in e.g. \cite{Carrillo1,Carrillo(2016)}.

In this paper, we consider a relaxation limit in the spirit of Jin-Xin \cite{JinXin} of the aggregation equation in one space dimension on $\RR$.
It is now well-established that such modifications allow to regularize the solutions.
For a given $c> \| a \|_\infty$, we introduce the system
\begin{subequations}\label{eq:agg_relax}
  \begin{align}
    & \pa_t\rho + \pa_x \sigma = 0,   \label{eq:rho} \\
    & \pa_t \sigma + c^2 \pa_x \rho = \frac{1}{\eps} (a[\rho] \rho - \sigma)   \label{eq:p} \\
    & a[\rho] = - W' * \rho  \label{eq:arho}
  \end{align}
\end{subequations}
This system is complemented with initial data $\rho_0$ and $\sigma_0 := a[\rho_0] \rho_0$.
It is clear, at least formally, that when $\eps\to 0$ the solution $\rho$ of system \eqref{eq:agg_relax} converges to the one of the aggregation equation \eqref{eq:aggreg} (and actually it is true only if $c > \| a \|_\infty$).
We mention that the aggregation equation may also be derived thanks to a hydrodynamical limit of kinetic equations \cite{DolakSchmeiser,JamesVauchelet,JV_sinum}.

The aim of this work is to study the convergence as $\eps\to 0$ of the relaxation system \eqref{eq:agg_relax} towards the aggregation equation. More precisely, we establish a precise estimate of the speed of convergence, and we also illustrate with some numerical simulations.
These estimates are obtained only in the case of the Newtonian potential in one dimension $W(x) = \frac 12 |x|$.
Indeed, in this particular case we may link the aggregation equation to a scalar conservation law \cite{Bonaschi,JV_dcds}. The same link holds for the relaxation system \eqref{eq:agg_relax}: denoting
$$
u(t,x) = \frac 12 - \int_{-\infty}^x \rho(t,dy), \qquad
v(t,x) = \frac 12 - \int_{-\infty}^x \sigma(t,dy),
$$
where the notation $\int \rho(t,dy)$ stands for the integral with respect to the probability measure $\rho(t)$, then we verify easily that
$$
u = - W'*\rho, \qquad \rho = - \pa_x u,
$$
so that $a[\rho] = u$.
Then, integrating \eqref{eq:agg_relax}, we deduce that $(u,v)$ is a solution to
\begin{subequations}\label{eq:uv}
  \begin{align}
    & \pa_t u + \pa_x v = 0  \label{eq:u}  \\
    & \pa_t v + c^2 \pa_x u = \frac{1}{\eps} \Big(\frac 12 u^2 - v\Big),   \label{eq:v}
  \end{align}
\end{subequations}
which is complemented with initial data $u_0 = \frac 12 - \int_{-\infty}^x \rho_0(dy)$, and $v_0 = \frac 12 - \int_{-\infty}^x \sigma_0(dy)$.
Clearly, as $\eps\to 0$, we expect that the solution of the above system converges to the solution of the following Burgers equation
$$
\pa_t u + \frac 12 \pa_x u^2 = 0.
$$
Introducing the quantities $a=v-cu$ and $b=v+cu$, \eqref{eq:uv} is equivalent to the diagonalized system
\begin{subequations}\label{eq:ab}
  \begin{align}
    &\pa_t a - c\pa_x a = \frac{1}{\eps} \Big(\frac 12 \Big(\frac{b-a}{2c}\Big)^2 - \frac{a+b}{2}\Big)  \label{eq:a}  \\
    &\pa_t b + c\pa_x b = \frac{1}{\eps} \Big(\frac 12 \Big(\frac{b-a}{2c}\Big)^2 - \frac{a+b}{2}\Big).  \label{eq:b}
  \end{align}
\end{subequations}
We will adapt the techniques developed in \cite{Katsoulakis(1997)} to obtain convergence estimates for our system.

In order to illustrate this convergence result, numerical discretizations of the relaxation system \eqref{eq:agg_relax} are investigated. The schemes we propose are such that they are uniform with respect to $\eps$, that is they satisfy the so-called asymptotic preserving (AP) property \cite{ShiJin}.
Therefore, such schemes in the limit $\eps\to 0$ must be consistent with the aggregation equation.
Numerical simulations of solutions of the aggregation equation for pointy potentials have been studied by several authors see e.g. \cite{JV_sinum,Chertock,Craig,Carrillo(2016),Gosse(2016),Fabreges(2019),secondorder}.
In particular, some authors pay attention to recover the correct behavior of the numerical solutions after the blow-up time. To do so, particular attention must be paid to the definition of the product $a[\rho] \rho$ when $\rho$ is a measure.

In this article, we propose two discretizations of the relaxation system which satisfy the AP property. In a first approach, we propose a simple splitting algorithm where we split the transport part and the right hand side in system \eqref{eq:agg_relax}.
It results in a numerical scheme which is very simple to implement and for which we verify easily the AP property.
The second approach relies on a well-balanced discretization in the spirit of \cite{Gosse_book,Gosse(2016)}.
This scheme is more expensive to implement than the first scheme, but its numerical solution has less diffusion, as it is illustrated by our numerical results.

The outline of the paper is the following.
In section \ref{sec:conv}, after recalling some useful notations, we prove our main result: an estimation of the speed of convergence in the Wasserstein $W_1$ distance with respect to $\eps$ of the solutions of the relaxation system \eqref{eq:agg_relax} towards the solution of the aggregation equation \eqref{eq:aggreg} in the case $W(x)=\frac 12 |x|$.
The numerical discretization is investigated in section \ref{sec:num}.
Two numerical schemes verifying the AP property are proposed. The first scheme is based on a splitting algorithm, whereas the second scheme relies on a well-balanced discretization.
Numerical results and comparisons are provided in section \ref{sec:numresult}.

\section{Convergence result}\label{sec:conv}

\subsection{Notations}

Before stating and proving our main results, we first recall some useful notations and results.
Since we are dealing with conservation laws (in which the total mass is conserved), we will work in some space of probability measures, namely the Wasserstein space of order $p\geq 1$, which is the space of probability measures with finite order $p$ moment:
    $$
\calP_p(\RR^N) = \left\{\mu \mbox{ nonnegative Borel measure}, \mu(\RR^N)=1, \int |x|^p \mu(dx) <\infty\right\}.
    $$
This space is endowed with the Wasserstein distance defined by (see e.g. \cite{Villani, Santambrogio})
\begin{equation}
  \label{defWp}
  W_p(\mu,\nu)= \inf_{\gamma\in \Gamma(\mu,\nu)} \left\{\int |y-x|^p\,\gamma(dx,dy)\right\}^{1/p},
\end{equation}
where $\Gamma(\mu,\nu)$ is the set of measures on $\RR^N\times\RR^N$ with marginals $\mu$ and $\nu$, i.e.
$$
\begin{array}{c}
  \Gamma(\mu,\nu) = \displaystyle\left\{\gamma\in \calP_p(\RR^N\!\times\!\RR^N); \
  \forall\, \xi\in C_0(\RR^N),
\int_{\RR^{2N}} \xi(y_0)\gamma(dy_0,dy_1) = \int_{\RR^N} \xi(y_0) \mu(dy_0), \right.\\[2mm]
\displaystyle\left.\int_{\RR^{2N}} \xi(y_1)\gamma(dy_0,dy_1) = \int_{\RR^N} \xi(y_1) \nu(dy_1) \right\},
\end{array}
$$
with $C_0(\RR^N)$ the set of continuous functions on $\RR^N$ that vanish at infinity.
From a simple minimization argument, we know that in the definition of $W_p$
the infimum is actually a minimum. A map that realizes the minimum in the definition \eqref{defWp} of $W_p$ is called an optimal transport plan, the set of which is denoted by $\Gamma_0(\mu,\nu)$.

In the one-dimensional framework, we may simplify these definitions.
Indeed any probability measure $\mu$ on the real line $\RR$ can be described in term of its cumulative distribution function $F_\mu(x)=\mu((-\infty,x))$ which is a right-continuous and nondecreasing function with $F_\mu(-\infty)=0$ and $F_\mu(+\infty)=1$.
Then we can define the generalized inverse $F_\mu^{-1}$ of $F_\mu$ (or monotone rearrangement of $\mu$) by $F_\mu^{-1}(z):=\inf\{x\in \RR / F_\mu(x)>z\}$, it is a right-continuous and nondecreasing function as well, defined on $[0,1]$.
We have for every nonnegative Borel map $\xi$,
$$
\int_\RR \xi(x) \mu(dx) = \int_0^1 \xi(F_\mu^{-1}(z))\,dz.
$$
In particular, $\mu\in \calP_p(\RR)$ if and only if $F_\mu^{-1}\in L^p(0,1)$.
Moreover, in the one-dimensional setting, there exists a unique optimal transport plan realizing the minimum in \eqref{defWp}.
More precisely, if $\mu$ and $\nu$ belong to $\calP_p(\RR)$, with monotone rearrangements $F_\mu^{-1}$ and $F_\nu^{-1}$, then $\Gamma_0(\mu,\nu)=\{(F_\mu^{-1},F_\nu^{-1})_\# {\mathbb{L}}_{(0,1)}\}$ where ${\mathbb{L}}_{(0,1)}$ is the restriction of the Lebesgue measure on $(0,1)$.
Thus we have the explicit expression of the Wasserstein distance (see \cite{Vallender(1974),Rachev,Villani})
\begin{equation}\label{dWF-1}
  W_p(\mu,\nu) = \left(\int_0^1 |F_\mu^{-1}(z)-F_\nu^{-1}(z)|^p\,dz\right)^{1/p},
\end{equation}
and the map $\mu \mapsto F_\mu^{-1}$ is an isometry between $\calP_p(\RR)$ and the convex subset of (essentially) nondecreasing functions of $L^p(0,1)$.

\subsection{Convergence estimates}

Let us first consider the limit $\eps \to 0$ for the system \eqref{eq:uv}. Compactness methods have been used in \cite{Natalini(1996)} to get $L^1_{loc}$ convergence in space. However, in order to pass to the aggregation equation, one may want global $L^1$ convergence, which we prove in the following theorem, along the lines of \cite{Katsoulakis(1997)}:

\begin{theorem}\label{thm:cvgu}
Let $u_0 \in L^\infty \cap BV(\RR)$, $c > \|u_0\|_{L^\infty}$ and set $v_0 = \frac{u_0^2}{2}$. There exists a constant $C > 0$ such that, for any $\eps > 0$, denoting by $(u^\eps, v^\eps)$ the solution to \eqref{eq:uv} with initial data $(u_0, v_0)$, the following estimate holds:
$$
\forall T > 0, \qquad \|u(T) - u^\eps(T)\|_{L^1} \leq C TV(u_0) (\sqrt{\eps T} + \eps),
$$
where $u$ is the entropy solution to the Burgers equation with initial datum $u_0$.
\end{theorem}

\begin{proof}
Denote $(a^\eps, b^\eps)$ the solution to \eqref{eq:ab}, and $G(a,b) = \frac 12 \Big( \frac{b-a}{2c} \Big)^2 - \frac{a+b}{2}$.

So as to obtain entropy inequalities on $(a^\eps, b^\eps)$, we need monotonicity properties on $G$. One can check that $G(a^\eps,b^\eps)$ is decreasing with respect to $a^\eps$ and $b^\eps$ if the so-called subcharacteristic condition $|u^\eps| < c$ holds. Up to a slight modification of the nonlinear term $f(u^\eps) = \frac{(u^\eps)^2}{2}$ in \eqref{eq:uv}, which does not affect the value of $(a^\eps, b^\eps)$:

$$
f(u) :=
\begin{cases}
- \|u_0\| u - \dfrac{\|u_0\|^2}{2}, \quad & \text{if } u \leq -\|u_0\|, \\
\dfrac{u^2}{2}, \quad & \text{if } - \|u_0\| \leq u \leq \|u_0\|, \\
\|u_0\| u - \dfrac{\|u_0\|^2}{2}, \quad & \text{if } \|u_0\| \leq u,
\end{cases}
$$
the choice $c > \|u_0\|_{L^\infty}$ ensures that the subcharacteristic condition and the bound $\|u^\eps(t)\|_{L^\infty} \leq \|u_0\|_{L^\infty}$ hold for all time.

\ms

Now, obtaining entropy inequalities on $(a^\eps, b^\eps)$ consists in making a comparison with constant state solutions to \eqref{eq:ab}. Namely, letting $m = \|u_0\|_{L^\infty} \Big( \frac{\|u_0\|_{L^\infty}}{2} - c \Big)$, $M = \|u_0\|_{L^\infty} \Big( \frac{\|u_0\|_{L^\infty}}{2} + c \Big)$ and $h(a) = a + 2c^2 - 2c \sqrt{c^2 + 2a}$, we have $G(k, h(k)) = 0$ for all $k \in [m, M]$, and therefore $(k, h(k))$ is a solution to \eqref{eq:ab}. Thus the following system holds:

\begin{subequations}\label{eq:kh}
  \begin{align}
    & \pa_t (a^\eps - k) - c \pa_x (a^\eps - k) = \frac{1}{\eps} \Big( G(a^\eps, b^\eps) - G(k,h(k)) \Big),  \label{eq:k}  \\
    & \pa_t (b^\eps - h(k)) + c \pa_x (b^\eps - h(k)) = \frac{1}{\eps} \Big( G(a^\eps, b^\eps) - G(k,h(k)) \Big).  \label{eq:h}
  \end{align}
\end{subequations}
Multiplying \eqref{eq:k} by $\sgn(a^\eps - k)$, \eqref{eq:h} by $\sgn(b^\eps - h(k))$ and summing yields:
\begin{align*}
\pa_t \Big( |a^\eps - k| + |b^\eps - h(k)| \Big) - c \pa_x \Big( |a^\eps - k| - |b^\eps - h(k)| \Big) = \frac{1}{\eps} \Big( & \sgn(a^\eps - k) + \sgn(b^\eps - h(k)) \Big) \times \\
& \Big( G(a^\eps, b^\eps) - G(k,h(k)) \Big).
\end{align*}
Hence, using the monotonicity of $G$ we get the following entropy inequalities on $(a^\eps, b^\eps)$:
\begin{equation}\label{entropie:ab}
\pa_t \Big( |a^\eps - k| + |b^\eps - h(k)| \Big) - c \pa_x \Big( |a^\eps - k| - |b^\eps - h(k)| \Big) \leq 0.
\end{equation}
We now turn to proving entropy inequalities on $u^\eps$. Straightforward computations yield the existence of a constant $C > 0$ such that, for all $a,b \in [m, M]$, one has $|h(a) - b| \leq C |G(a,b)|$. We therefore work on the variable $w^\eps := \frac{h(a^\eps) - a^\eps}{2c}$ in the first place. Let $\kappa \in \big[ -\|u_0\|_{L^\infty}, \|u_0\|_{L^\infty} \big]$, and $k \in [m, M]$ such that $\kappa = \frac{h(k) - k}{2c}$. We have:
\begin{equation}\label{eq:w1}
|w^\eps - \kappa| = \frac{1}{2c} \Big( |h(a^\eps) - h(k)| + |a^\eps - k| \Big) = \frac{1}{2c} \Big( |a^\eps - k| + |b^\eps - h(k)| + r_1^\eps \Big),
\end{equation}
where $r_1^\eps = |h(a^\eps) - h(k)| - |b^\eps - h(k)|$ verifies $|r_1^\eps| \leq |h(a^\eps) - b^\eps| \leq C |G(a^\eps, b^\eps)|$. Thus, we are left to control $|G(a^\eps, b^\eps)|$. To do so, we formally differentiate this quantity and use \eqref{eq:ab}:
\begin{align*}
\pa_t |G(a^\eps, b^\eps)| & =  \Big( \pa_t a^\eps \pa_a G(a^\eps, b^\eps) + \pa_t b^\eps \pa_b G(a^\eps, b^\eps) \Big) \sgn(G(a^\eps, b^\eps)), \\
& = \frac{1}{\eps} \Big( \pa_a G(a^\eps, b^\eps) + \pa_b G(a^\eps, b^\eps) \Big) |G(a^\eps, b^\eps)| - c \sgn(G(a^\eps, b^\eps)) \Big( \pa_x a^\eps \pa_a G(a^\eps, b^\eps) + \pa_x b^\eps \pa_b G(a^\eps, b^\eps) \Big), \\
& \leq \frac{1}{\eps} \sup_{[m, M]^2} \Big(\pa_a G + \pa_b G \Big) |G(a^\eps, b^\eps)| + c \sup_{[m, M]^2} \Big(|\pa_a G| + |\pa_b G| \Big) \Big(|\pa_x a^\eps| + |\pa_x b^\eps| \Big).
\end{align*}
Integrating in space gives:
$$
\frac{\ud}{\ud t} \|G(a^\eps, b^\eps)\|_{L^1} \leq - \frac{A}{\eps} \|G(a^\eps, b^\eps)\|_{L^1} + B \Big( TV(a_0) + TV(b_0) \Big),
$$
where $A = -\sup_{[m, M]^2} (\pa_a G + \pa_b G)$ and $B = c \sup_{[m, M]^2} (|\pa_a G| + |\pa_b G|)$ are positive constants which do not depend on $\eps$ nor on time. A Gronwall lemma then gives:
\begin{equation}\label{eq:G}
\|G(a^\eps(t), b^\eps(t))\|_{L^1} \leq C \Big( TV(a_0) + TV(b_0) \Big) \eps,
\end{equation}
where we still denote $C = B/A$ a constant independent of time and of $\eps$.

Besides, since, $G(a, h(a)) = 0$, one has $\frac{1}{2} \left(\frac{h(a)-a}{2c}\right)^2 = \frac{1}{2} (h(a) + a)$ and therefore:
\begin{align}\label{eq:w2}
\sgn(w^\eps - \kappa) \left(\frac{(w^\eps)^2}{2} - \frac{\kappa^2}{2} \right) & = \frac{1}{2} \sgn \Big(h(a^\eps) - h(k) - (a^\eps - k) \Big) \Big( h(a^\eps) + a^\eps - (h(k) + k) \Big), \nonumber \\
& = \frac{1}{2} \Big( |h(a^\eps) - h(k)| - |a^\eps - k| \Big), \nonumber \\
& = \frac{1}{2} \Big( |b - h(k)| - |a^\eps - k| + r_2^\eps \Big),
\end{align}
with $|r_2^\eps| \leq C |G(a^\eps, b^\eps)|$. Differentiating \eqref{eq:w1} in time and \eqref{eq:w2} in space, and using \eqref{entropie:ab} thus yields:
\begin{equation}\label{entropie:w}
\pa_t |w^\eps - \kappa| + \pa_x \sgn(w^\eps - \kappa) \left(\frac{(w^\eps)^2}{2} - \frac{\kappa^2}{2} \right) \leq \frac{1}{2c} \Big( \pa_t r_1^\eps + c \pa_x r_2^\eps \Big).
\end{equation}
Then, we estimate $\|u(t) - w^\eps(t)\|_{L^1}$ using Kuznetsov's doubling of variables technique (see e.g. \cite{Serre(1999)} for scalar conservation laws with viscosity and \cite{Bouchut(1998)} for a more general formalism) in order to combine \eqref{entropie:w} with Kruzkov inequalities on the entropy solution $u$, that read:
\begin{equation}\label{entropie:u}
\pa_t |u - \kappa| + \pa_x \sgn(u - \kappa) \left(f(u) - f(\kappa)\right) \leq 0.
\end{equation}
Writing respectively \eqref{entropie:u} at point $(s,x)$ for $\kappa = w^\eps(t,y)$ and \eqref{entropie:w} at point $(t,y)$ for $\kappa = u(s,x)$, we get:
\begin{subequations}\label{doubling}
\begin{align}
& \pa_s |u(s,x) - w^\eps(t,y)| + \pa_x \sgn(u(s,x) - w^\eps(t,y)) \left(\frac{u(s,x)^2}{2} - \frac{(w^\eps(t,y))^2}{2} \right) \leq 0, \label{doubling1} \\
& \pa_t |w^\eps(t,y) - u(s,x)| + \pa_y \sgn(w^\eps(t,y) - u(s,x)) \left(\frac{(w^\eps(t,y))^2}{2} - \frac{u(s,x)^2}{2} \right) \leq \frac{1}{2c} \Big( \pa_t r_1^\eps(t,y) + c \pa_y r_2^\eps(t,y) \Big). \label{doubling2}
\end{align}
\end{subequations}
Now, let $\omega_\alpha(t) = \frac{1}{\alpha} \omega \left(\frac{t}{\alpha}\right)$ and $\Omega_\beta(x) = \frac{1}{\beta} \Omega \left(\frac{x}{\beta}\right)$ be two mollyfing kernels. Setting $g(s, t, x, y) = \omega_\alpha(s - t) \Omega_\beta(x - y)$ and testing \eqref{doubling1} and \eqref{doubling2} against $g(\cdot, t, \cdot, y) \mathbf{1}_{[0,T]}$ and $g(s, \cdot, x, \cdot) \mathbf{1}_{[0,T]}$ respectively, and integrating over $[0,T] \times \RR$, we get on the one hand:
\begin{align}\label{doubling3}
& \intquad \pa_s g(s, t, x, y) |u(s,x) - w^\eps(t,y)| \ud s \ud x \ud t \ud y \nonumber \\
& + \intquad \pa_x g(s, t, x, y) \sgn(u(s,x) - w^\eps(t,y)) \left(\frac{u(s,x)^2}{2} - \frac{(w^\eps(t,y))^2}{2} \right) \ud s \ud x \ud t \ud y \nonumber \\
& - \inttriple g(T, t, x, y) |u(T,x) - w^\eps(t,y)| \ud x \ud t \ud y + \inttriple g(0, t, x, y) |u(0,x) - w^\eps(t,y)| \ud x \ud t \ud y \geq 0,
\end{align}
and on the other hand:
\begin{align}\label{doubling4}
& \intquad \pa_t g(s, t, x, y) |w^\eps(t,y) - u(s,x)| \ud s \ud x \ud t \ud y \nonumber \\
& + \intquad \pa_y g(s, t, x, y) \sgn(w^\eps(t,y) - u(s,x)) \left(\frac{(w^\eps(t,y))^2}{2} - \frac{u(s,x)^2}{2} \right) \ud s \ud x \ud t \ud y \nonumber \\
& - \inttriple g(s, T, x, y) |w^\eps(T,y) - u(s,x)| \ud s \ud xd \ud y + \inttriple g(s, 0, x, y) |w^\eps(0,y) - u(s,x)| \ud s \ud x \ud y \nonumber \\
& \geq \frac{1}{2c} \bigg( \intquad \pa_t g(s, t, x, y) r_1^\eps(t,y) \ud s \ud x \ud t \ud y + c  \intquad \pa_y g(s, t, x, y) r_2^\eps(t,y) \ud s \ud x \ud t \ud y \nonumber \\
& - \inttriple g(s, T, x, y) r_1^\eps(T, y) \ud s \ud x \ud y + \inttriple g(s, 0, x, y) r_1^\eps(0,y) \ud s \ud x \ud y \bigg) =: \text{RHS}.
\end{align}
Now, since $|\cdot|$ is even, and $\pa_s g = - \pa_t g$ and $\pa_x g = - \pa_y g$, we deduce by adding \eqref{doubling3} and \eqref{doubling4}:
\begin{align}\label{doubling5}
& - \inttriple g(T, t, x, y) |u(T,x) - w^\eps(t,y)| \ud x \ud t \ud y + \inttriple g(0, t, x, y) |u(0,x) - w^\eps(t,y)| \ud x \ud t \ud y \nonumber \\
& - \inttriple g(s, T, x, y) |u(s,x) - w^\eps(T,y)| \ud s \ud x \ud y + \inttriple g(s, 0, x, y) |u(s,x) - w^\eps(0,y)| \ud s \ud x \ud y \geq \text{RHS}.
\end{align}
Then, we write:
\begin{align}\label{eq:I12}
\|u(T) - w^\eps(T)\|_{L^1} & = \inttriple \omega_\alpha(T - t) \Omega_\beta(x - y) |u(T,y) - w^\eps(T,y)| \ud x \ud t \ud y \nonumber \\
& + \inttriple \omega_\alpha(s - T) \Omega_\beta(x - y) |u(T,y) - w^\eps(T,y)| \ud s \ud x \ud y, \nonumber \\
& =: I_1 + I_2.
\end{align}
A triangle inequality gives for $I_1$:
\begin{align*}
I_1 & \leq \inttriple \omega_\alpha(T - t) \Omega_\beta(x - y) |u(T,y) - u(T,x)| \ud x \ud t \ud y \\
& + \inttriple \omega_\alpha(T - t) \Omega_\beta(x - y) |u(T,x) - w^\eps(t,y)| \ud x \ud t \ud y \\
& + \inttriple \omega_\alpha(T - t) \Omega_\beta(x - y) |w^\eps(t,y) - w^\eps(T,y)| \ud x \ud t \ud y \\
& =: T_1 + T_2 + T_3.
\end{align*}
with $T_1 \leq C \beta \cdot TV(u_0)$, the second term $T_2$ appearing in \eqref{doubling5} and for the last one we write:
$$
T_3 \leq \int_\RR \Omega_\beta(x - y) \int_0^T \omega_\alpha(T - t) \int_\RR |w^\eps(t,y) - w^\eps(T,y)| \ud y  \ud t \ud x,
$$
and then we use the fact that $w^\eps$ is uniformely Lipschitz in $L^1(\RR)$ with respect to $\eps$. Indeed, one has $\pa_t w^\eps = \frac{\pa_t a^\eps (h'(a^\eps) - 1)}{2c}$ with $h'(a^\eps) - 1$ being uniformely bounded with respect to $\eps$ as $a^\eps$ stays in the compact set $[m,M]$ for all time. In addition, estimating $\|\pa_t a^\eps(t) \|_{L^1}$ can be done reusing \eqref{eq:ab} and \eqref{eq:G}:
$$
\|\pa_t a^\eps(t) \|_{L^1} \leq c \|\pa_x a^\eps(t) \|_{L^1} + \frac{1}{\eps} \|G(a^\eps(t), b^\eps(t))\|_{L^1} \leq C \left(TV(a_0) + TV(b_0)\right).
$$
with $C > 0$ still independent of time and of $\eps$. Hence $\|\pa_t w^\eps(t) \|_{L^1} \leq C \left(TV(a_0) + TV(b_0)\right)$ and $T_3 \leq \alpha C \left(TV(a_0) + TV(b_0)\right)$. All in all, we get for $I_1$:
$$
I_1 \leq \inttriple \omega_\alpha(T - t) \Omega_\beta(x - y) |u(T,x) - w^\eps(t,y)| \ud x \ud t \ud y +  C \beta \cdot TV(u_0) + \alpha C \left(TV(a_0) + TV(b_0)\right).
$$
And, similarly, for $I_2$:
$$
I_2 \leq \inttriple \omega_\alpha(s - T) \Omega_\beta(x - y) |u(s,x) - w^\eps(T,y)| \ud s \ud x \ud y + C(\alpha + \beta)TV(u_0).
$$
Back to \eqref{eq:I12}, we obtain:
\begin{align}\label{doubling6}
\|u(T) - w^\eps(T)\|_{L^1} & \leq \inttriple \omega_\alpha(t) \Omega_\beta(x - y) |u(0,x) - w^\eps(t,y)| \ud x \ud t \ud y \\
& + \inttriple \omega_\alpha(s) \Omega_\beta(x - y) |u(s,x) - w^\eps(0,y)| \ud s \ud x \ud y - RHS \\
& + \alpha C \left(TV(a_0) + TV(b_0)\right) +  C(\alpha + \beta)TV(u_0).
\end{align}
But using a triangle inequality, one can show that:
$$
\inttriple \omega_\alpha(t) \Omega_\beta(x - y) |u_0(x) - w^\eps(t,y)| \ud x \ud t \ud y \leq C \beta \cdot TV(u_0) + \alpha C \left(TV(a_0) + TV(b_0)\right),
$$
and similarly:
$$
\inttriple \omega_\alpha(s) \Omega_\beta(x - y) |u(s,x) - w^\eps(0,y)| \ud s \ud x \ud y \leq C (\alpha + \beta) TV(u_0).
$$
We then bound from above the term RHS using inequality $\|r_i^\eps(t) \|_{L^1} \leq C \left(TV(a_0) + TV(b_0)\right) \eps$ for $i = 1,2$:
\begin{align*}
\bigg| \text{RHS} \bigg| & = \frac{1}{2c} \bigg| \frac{1}{\alpha} \intquad \omega'\left(\frac{s - t}{\alpha}\right) \Omega_\beta(x - y) r_1^\eps(t,y) \ud s \ud x \ud t \ud y + \frac{c}{\beta} \intquad \omega_\alpha(s - t) \Omega'(x - y) r_2^\eps(t,y) \ud s \\
&  \ud x \ud t \ud y - \inttriple \omega_\alpha(s - T) \Omega_\beta(x - y) r_1^\eps(T, y) \ud s \ud x \ud y + \inttriple \omega_\alpha(s) \Omega_\beta(x - y) r_1^\eps(0,y) \ud s \ud x \ud y \bigg|, \\
& \leq C \left(\frac{T}{\alpha} + \frac{T}{\beta} + 1 \right) \cdot \left(TV(a_0) + TV(b_0)\right) \eps.
\end{align*}
Finally, we get:
$$
\|u(T) - w^\eps(T)\|_{L^1} \leq C \left(\frac{T}{\alpha} + \frac{T}{\beta} + 1 \right) \left(TV(a_0) + TV(b_0)\right) \eps + C (\alpha + \beta) TV(u_0) + \alpha C \left(TV(a_0) + TV(b_0)\right),
$$
which, after optimizing the values of $\alpha$ and $\beta$ and noticing that $TV(a_0), TV(b_0) \leq C \cdot TV(u_0)$, gives:
$$
\|u(T) - w^\eps(T)\|_{L^1} \leq C TV(u_0) (\sqrt{\eps T} + \eps),
$$
and this inequality, along with $|h(a) - b| \leq C |G(a,b)|$ and \eqref{eq:G} gives in turn the result.
\end{proof}

Denoting $\rho = - \pa_x u$, the convergence of $u^\eps(t)$ towards $u(t)$ in $L^1(\RR)$ ensures that $\rho(t)$ is a probability measure. Indeed, since for all $\eps > 0$, $\rho^\eps = - \pa_x u^\eps$ is a nonnegative distribution, so is $\rho$. The Riesz-Markov theorem then ensures that $\rho$ can be represented by a nonnegative Borel measure. Besides, for a.e. $t \geq 0$, $u^\eps(t)$ is a nonincreasing function taking values in $[0,1]$ and hence converges to a certain limit when $x$ goes to $+\infty$. The same holds true for the limit function $u(t)$. But, since $u^\eps(t) - u(t) \in L^1(\RR)$, then  $u^\eps(t,x) - u(t,x)$ must vanish as $x$ goes to $+\infty$. Therefore the total mass of $\rho(t)$ is 1.

 Then, passing to the relaxation system \eqref{eq:agg_relax} for the aggregation equation \eqref{eq:aggreg} can be done by using \eqref{dWF-1} with $p=1$.
As a consequence, Theorem \ref{thm:cvgu} translates as follows for the aggregation:
\begin{theorem}\label{thm:cvgrho}
Let $\rho_0 \in \calP_2(\RR)$, $c > 1/2$ and set $\sigma_0 = a[\rho_0] \rho_0$. There exists a constant $C > 0$ such that, for any $\eps > 0$, denoting $(\rho^\eps, \sigma^\eps)$ the solution to \eqref{eq:agg_relax} with initial data $(\rho_0, \sigma_0)$, one has :
$$
\forall T > 0, \qquad W_1(\rho(T),\rho^\eps(T)) \leq C (\sqrt{\eps T} + \eps),
$$
where $\rho \in C([0, +\infty), \calP_2(\RR))$ is the unique solution \eqref{eq:aggreg} with initial datum $\rho_0$.
\end{theorem}

\section{Numerical discretization}\label{sec:num}

From now on, we denote $\Delta t$ the time step and we introduce a Cartesian mesh of size $\Delta x$. We denote $t^n=n\Delta t$ for $n\in\NN$ and $x_j = j\Delta x$ for $j\in\ZZ$. In this section, we extend our framework and consider the aggregation equation \eqref{eq:aggreg} with arbitrary pointy potentials $W$, which satisfy the following conditions:
\begin{enumerate}
\item $W$ is even and $W(0) = 0$,
\item $W \in \mathcal{C}^1(\RR \setminus \{0\})$,
\item $W$ is $\lambda$-convex, i.e. there exists $\lambda \in \RR$ such that $W(x) - \lambda \frac{|x|^2}{2}$ is convex,
\item $W$ is $a_\infty$-lipschitz continuous for some $a_\infty \geq 0$.
\end{enumerate}
In this framework, the convergence of $\rho^\eps$ towards $\rho$ for a slightly different problem has also been studied in \cite{JamesVauchelet}. Adapting the argument, the convergence still holds provided the subcharacteristic condition $c > a_\infty$ is verified.
However, for such general potentials, the authors were not able to obtain the estimates of the speed of convergence as stated in Theorem \ref{thm:cvgrho}.

In this section, we propose some numerical schemes able to capture the limit $\eps\to 0$, that is satisfying the so-called asymptotic preserving (AP) property.
We consider two approaches, the first one based on a splitting algorithm, the second one based on a well-balanced discretisation.

\subsection{A splitting algorithm}

A first simple approach to discretize system \eqref{eq:agg_relax} is to use a splitting method.
Such a method is known to be convergent and easy to implement but introduces numerical diffusion.

Notice that the system \eqref{eq:agg_relax} rewrites, with $\mu = \sigma - c\rho$, $\nu = \sigma + c\rho$, as:
\begin{subequations}\label{eq:munu}
  \begin{align}
    &\pa_t \mu - c\pa_x \mu = \frac{1}{\eps} \Big(a\Big[\frac{\nu - \mu}{2c}\Big] \Big(\frac{\nu - \mu}{2c}\Big) - \frac{\mu + \nu}{2}\Big)  \label{eq:mu}  \\
    &\pa_t \nu + c\pa_x \nu = \frac{1}{\eps} \Big(a\Big[\frac{\nu - \mu}{2c}\Big] \Big(\frac{\nu - \mu}{2c}\Big) - \frac{\mu + \nu}{2}\Big).  \label{eq:nu}
  \end{align}
\end{subequations}
The idea of the method is to solve in a first step on $(t^n,t^n+\Delta t)$ the system
\begin{align*}
  &\pa_t \mu = \frac{1}{\eps} \Big(a\Big[\frac{\nu - \mu}{2c}\Big] \Big(\frac{\nu - \mu}{2c}\Big) - \frac{\mu + \nu}{2}\Big)  \\
  &\pa_t \nu =  \frac{1}{\eps} \Big(a\Big[\frac{\nu - \mu}{2c}\Big] \Big(\frac{\nu - \mu}{2c}\Big) - \frac{\mu + \nu}{2}\Big),
\end{align*}
with initial data $(\mu(t^n),\nu(t^n))=(\mu^n,\nu^n)$.
We obtain $\mu^{n+\frac 12}_j = \mu(t^n+\Delta t,x_j)$ and $\nu^{n+\frac 12}_j = \nu(t^n+\Delta t,x_j)$.
Notice that this system may be solved explicitely. Indeed, by adding and subtracting the two equations, we deduce after an integration
\begin{subequations}\label{split1}
\begin{align}
  &\nu^{n+\frac 12}_j - \mu^{n+\frac 12}_j = \nu^n_j - \mu^n_j  \label{split1:mu} \\
  &\mu^{n+\frac 12}_j + \nu^{n+\frac 12}_j = (\mu^n_j+\nu^n_j) e^{-\Delta t/\eps} + a\Big[\frac{\nu^n - \mu^n}{2c}\Big] \Big(\frac{\nu^n - \mu^n}{2c}\Big) (1-e^{-\Delta t/\eps}).  \label{split1:nu}
\end{align}
\end{subequations}
Then, in a second step, we discretize by a classical finite volume upwind scheme the system
\begin{align*}
  &\pa_t \mu - c\pa_x \mu = 0, \qquad
  \pa_t \nu + c\pa_x \nu = 0.
\end{align*}
That is
\begin{subequations}\label{split2}
\begin{align}
  & \mu_j^{n+1} = \mu_j^{n+\frac 12} + c \frac{\Delta t}{\Delta x} (\mu_{j+1}^{n+\frac 12} - \mu_j^{n+\frac 12}),   \label{split2:mu}  \\
  & \nu_j^{n+1} = \nu_j^{n+\frac 12} - c \frac{\Delta t}{\Delta x} (\nu_j^{n+\frac 12} - \nu_{j-1}^{n+\frac 12}).  \label{split2:nu}
\end{align}
\end{subequations}
Coming back to the variables $\rho$ and $\sigma$, we obtain
\begin{align*}
  & \nu_j^{n+\frac 12} = c \rho_j^n + \sigma_j^n e^{-\Delta x/\eps} + a_j^n \rho_j^n (1-e^{-\Delta t/\eps}),  \\
  & \mu_j^{n+\frac 12} = -c \rho_j^n + \sigma_j^n e^{-\Delta x/\eps} + a_j^n \rho_j^n (1-e^{-\Delta t/\eps}),
\end{align*}
with $a_j^n = - \ds \sum_{k \neq j} W'(x_j - x_k) \rho_k^n$. Then, the splitting algorithm reads
\begin{align}
  \rho_j^{n+1} & = \rho_j^n - \frac 12 \frac{\Delta t}{\Delta x} (\mu_{j+1}^{n+\frac 12} + \nu_j^{n+\frac 12} - \mu_{j}^{n+\frac 12} - \nu_{j-1}^{n+\frac 12})  \nonumber \\
            & = \rho_j^n - \frac 12 \frac{\Delta t}{\Delta x} \Big((\sigma_{j+1}^n-\sigma_{j-1}^n) e^{-\Delta t/\eps} + (1-e^{-\Delta t/\eps})(a_{j+1}^n \rho_{j+1}^n - a_{j-1}^n \rho_{j-1}^n) - c(\rho_{j+1}^n-2\rho_j^n+\rho_{j-1}^n) \Big),  \label{rho:split}
\end{align}
and
\begin{align}
  \sigma_j^{n+1} & = \sigma_j^{n+\frac 12} + \frac{c}{2} \frac{\Delta t}{\Delta x} (\sigma_{j+1}^n-2\sigma_j^n+\sigma_{j-1}^n)e^{-\Delta t/\eps}   \nonumber \\
            & \qquad + \frac{c}{2} \frac{\Delta t}{\Delta x}\left((a_{j+1}^n \rho_{j+1}^n - 2 a_j^n \rho_j^n+ a_{j-1}^n \rho_{j-1}^n)(1-e^{-\Delta t/\eps}) - c(\rho_{j+1}^n-\rho_{j-1}^n)\right)  \nonumber \\
            & = \sigma_j^n e^{-\Delta t/\eps} + a_j^n \rho_j^{n}(1-e^{-\Delta t/\eps}) +\frac{c}{2} \frac{\Delta t}{\Delta x} (\sigma_{j+1}^n-2\sigma_j^n+\sigma_{j-1}^n)e^{-\Delta t/\eps}   \label{sigma:split} \\
 & \qquad + \frac{c}{2} \frac{\Delta t}{\Delta x}\left((a_{j+1}^n \rho_{j+1}^n - 2 a_j^n \rho_j^n+ a_{j-1}^n \rho_{j-1}^n)(1-e^{-\Delta t/\eps}) - c(\rho_{j+1}^n-\rho_{j-1}^n)\right).   \nonumber
\end{align}
\begin{lemma}
For any $\eps > 0$, if both the CFL condition $\frac{c\Delta t}{\Delta x} \leq 1$ and the subcharacteristic condition $c \geq a_\infty$ hold, then the splitting scheme \eqref{split1}--\eqref{split2} is $L^1$-stable:
\begin{align*}
\forall n \in \NN, \qquad \sum_{j \in \ZZ} \left(|\mu_j^{n+1}| + |\nu_j^{n+1}|\right) \leq \sum_{j \in \ZZ} \left(|\mu_j^n| + |\nu_j^n|\right).
\end{align*}
\end{lemma}
\begin{proof}
We have:
\begin{align*}
\mu_j^{n+\frac 12} & = \frac 12 \left( e^{-\Delta t/\eps} \Big(1 + \frac{a_j^n}{c}\Big) + 1 - \frac{a_j^n}{c} \right) \mu_j^n - \frac{1 - e^{-\Delta t/\eps}}{2} \left(1 - \frac{a_j^n}{c}\right) \nu_j^n, \\
\nu_j^{n+\frac 12} & = - \frac{1 - e^{-\Delta t/\eps}}{2} \left(1 + \frac{a_j^n}{c}\right) \mu_j^n + \frac 12 \left( e^{-\Delta t/\eps} \Big(1 - \frac{a_j^n}{c}\Big) + 1 + \frac{a_j^n}{c} \right) \nu_j^n.
\end{align*}
Under the condition $c \geq a_{\infty}$, in the expression of $\mu_j^{n+\frac 12}$, the coefficient in front of $\mu_j^n$ is nonnegative and the one in front of $\nu_j^n$ is nonpositive. Similarly, in $\nu_j^{n+\frac 12}$, the coefficient of $\mu_j^n$ is nonpositive and the one in front of $\nu_j^n$ is nonnegative. Taking the absolute value and adding up therefore yields:
$$
\Big|\mu_j^{n+\frac 12}\Big| + \Big|\nu_j^{n+\frac 12}\Big| \leq \Big|\mu_j^n\Big| + \Big|\nu_j^n\Big|.
$$
It remains to remark that, provided the CFL condition $\frac{c\Delta t}{\Delta x} \leq 1$ is verified, \eqref{split2} gives:
\begin{align*}
\sum_{j \in \ZZ} \left(|\mu_j^{n+1}| + |\nu_j^{n+1}|\right) & \leq \left(1 - \frac{c\Delta t}{\Delta x}\right) \sum_{j \in \ZZ} \left(\Big|\mu_j^{n+\frac 12}\Big| + \Big|\nu_j^{n+\frac 12}\Big|\right) + \frac{c\Delta t}{\Delta x} \sum_{j \in \ZZ} \Big|\mu_{j+1}^{n+\frac 12}\Big| + \frac{c\Delta t}{\Delta x} \sum_{j \in \ZZ} \Big|\nu_{j-1}^{n+\frac 12}\Big|, \\
& \leq \left(1 - \frac{c\Delta t}{\Delta x}\right) \sum_{j \in \ZZ} \left(|\mu_j^n| + |\nu_j^n|\right) + \frac{c\Delta t}{\Delta x} \sum_{j \in \ZZ} \Big|\mu_j^{n+\frac 12}\Big| + \frac{c\Delta t}{\Delta x} \sum_{j \in \ZZ} \Big|\nu_j^{n+\frac 12}\Big|, \\
& \leq \sum_{j \in \ZZ} \left(|\mu_j^n| + |\nu_j^n|\right).
\end{align*}
\end{proof}
Note that similar schemes have also been studied in \cite{Liu(2000)} and proved convergent at rate $\sqrt{\Delta x}$.

\ss

Let us now verify the AP property.
When $\eps\to 0$, we verify that the equation on $\rho$ \eqref{rho:split} converges to the following Rusanov discretization of \eqref{eq:aggreg} (see \cite{Fabreges(2019)} for numerical simulations using the Rusanov scheme):
\begin{subequations}\label{dis:rusanov}
\begin{align}
& \rho_j^{n+1} =\rho_j^n - \frac 12 \frac{\Delta t}{\Delta x} \Big(a_{j+1}^n \rho_{j+1}^n - a_{j-1}^n \rho_{j-1}^n \Big) + \frac{c\Delta t}{2\Delta x} (\rho_{j+1}^n-2\rho_j^n+\rho_{j-1}^n), \\
& a_j^n = - \ds \sum_{k \neq j} W'(x_j - x_k) \rho_k^n.
\end{align}
\end{subequations}
This limiting scheme provides a consistant discretization  of \eqref{eq:aggreg}. Indeed, similar scheme has been extensively studied in \cite{Carrillo(2016)} using compactness arguments and the following convergence result has been proved:
\begin{lemma}
Assume $\rho_0 \in \mathcal{P}_2(\RR)$ and that the stability conditions $c \frac{\Delta t}{\Delta x} \leq 1$ and $c \geq a_\infty$ are satisfied. Let $T > 0$ and suppose we initialize the scheme \eqref{dis:rusanov} with $\rho_j^0 = \dfrac{1}{\Delta x} \rho_0(C_j)$ where $C_j = [x_{j-\frac 12}, x_{j+\frac 12})$. Then, denoting $\rho_{\Delta x}$ the reconstruction given by the scheme \eqref{dis:rusanov}, that is:
$$
\rho_{\Delta x}(t) = \sum_{n \in \NN} \sum_{j \in \ZZ} \rho_j^n \mathbf{1}_{[t^n, t^{n+1})}(t) \delta_{x_j},
$$
then $\rho_{\Delta x}$ converges weakly in the sense of measures on $[0,T] \times \RR$ towards the solution $\rho$ of equation \eqref{eq:aggreg} as $\Delta x$ goes to 0.
\end{lemma}
It has been also proved in \cite{Delarue} that the scheme \eqref{dis:rusanov} converges at rate $\sqrt{\Delta x}$.

\subsection{Well-balanced discretization}

Although the splitting method provides a simple way to obtain a discretization which is uniform with respect to the parameter $\eps$, the resulting scheme has strong numerical diffusion and may not have good large time behaviour.
Then, well-balanced schemes have been introduced.
A scheme is said to be well-balanced when it conserves equilibria. The method proposed in this section comes from \cite{Gosse(2016)}.

Let us assume that for some $n\in\NN$ the approximation $(\mu_j^n,\nu_j^n)_{j\in\ZZ}$ of $(\mu(t^n,x_j),\nu(t^n,x_j))_{j\in\ZZ}$ solution of \eqref{eq:munu} is known.
We construct an approximation at time $t^{n+1}$ using a finite volume upwind discretization of \eqref{eq:munu}, with the discretization of the source terms $H_{\mu,j}^n, H_{\nu,j}^n$ to be prescribed right afterwards:
\begin{subequations}\label{dis:upwind}
  \begin{align}
    & \mu_{j}^{n+1} = \mu_j^n + c \frac{\Delta t}{\Delta x} (\mu_{j+1}^n - \mu_j^n) + \frac{\Delta t}{\eps} H_{\mu,j}^n     \\
    & \nu_{j}^{n+1} = \nu_j^n - c \frac{\Delta t}{\Delta x} (\nu_j^n - \nu_{j-1}^n) + \frac{\Delta t}{\eps} H_{\nu,j}^n.
  \end{align}
\end{subequations}
In order to preserve equilibria, we set :
\begin{equation}\label{eq:source}
H_{\mu,j}^n = \frac{1}{\Delta x} \int_{x_{j-1}}^{x_j} H(\bar{\mu}, \bar{\nu}) \ud x, \qquad H(\mu,\nu) = a\Big[\frac{\nu - \mu}{2c}\Big] \Big(\frac{\nu - \mu}{2c}\Big) - \frac{\mu + \nu}{2},
\end{equation}
where $(\bar{\mu}, \bar{\nu})$ solve the stationary system with incoming boundary conditions, on $(x_{j-1}, x_j)$:
\begin{subequations}\label{eq:statbar}
  \begin{align}
    & - c \pa_x \bar{\mu} = \frac{1}{\eps} H(\bar{\mu}, \bar{\nu}) \label{eq:barmu}  \\
    & c \pa_x \bar{\nu} = \frac{1}{\eps} H(\bar{\mu}, \bar{\nu}) \label{eq:barnu}  \\
    & \bar{\mu}(x_j) = \mu_j^n, \qquad \bar{\nu}(x_{j-1}) = \nu_{j-1}^n.   \label{cond:statbar}
  \end{align}
\end{subequations}
And, in the same fashion, $H_{\nu,j}^n = \frac{1}{\Delta x} \int_{x_j}^{x_{j+1}} H(\tilde{\mu}, \tilde{\nu}) \ud x$, where $(\tilde{\mu}, \tilde{\nu})$ is the solution of the stationary system on $(x_j, x_{j+1})$:
\begin{subequations}\label{eq:stattilde}
  \begin{align}
    & - c \pa_x \tilde{\mu} = \frac{1}{\eps} H(\tilde{\mu}, \tilde{\nu}) \label{eq:tildemu}  \\
    & c \pa_x \tilde{\mu} = \frac{1}{\eps} H(\tilde{\mu}, \tilde{\nu}) \label{eq:tildenu}  \\
    & \tilde{\mu}(x_{j+1}) = \mu_{j+1}^n, \qquad \tilde{\nu}(x_j) = \nu_j^n,   \label{cond:stattilde}
  \end{align}
\end{subequations}
Reporting equations \eqref{eq:barnu} and \eqref{eq:tildemu} into the discretization of the source term, we get $H_{\nu,j}^n = \frac{c \eps}{\Delta x} (\bar{\nu}(x_j) - \nu_{j-1})$ and $H_{\mu,j}^n = - \frac{c \eps}{\Delta x} (\mu_j^n - \tilde{\mu}(x_j))$. Hence, one may rewrite the scheme \eqref{dis:upwind} as:
\begin{subequations}\label{dis:wb}
  \begin{align}
    & \mu_{j}^{n+1} = \mu_j^n + c \frac{\Delta t}{\Delta x} (\tilde{\mu}(x_j) - \mu_j^n)   \\
    & \nu_{j}^{n+1} = \nu_j^n - c \frac{\Delta t}{\Delta x} (\nu_j^n - \bar{\nu}(x_j)).
  \end{align}
\end{subequations}
Remark that the stationary system
\begin{equation}
- c \pa_x \mu = \frac{1}{\eps} H(\mu, \nu),  \qquad c \pa_x \nu = \frac{1}{\eps} H(\mu, \nu),
\end{equation}
is equivalent to
\begin{equation}
\pa_x \sigma = 0, \qquad c^2 \pa_x \rho = \frac{1}{\eps} \left(a[\rho] \rho - \sigma \right).
\end{equation}
 Therefore, denoting $\sigma_{j+\frac 12} = \dfrac{\tilde{\mu} + \tilde{\nu}}{2}$ and $\sigma_{j-\frac 12} = \dfrac{\bar{\mu} + \bar{\nu}}{2}$, which are constant respectively on $(x_j, x_{j+1})$ and $(x_{j-1}, x_j)$, one has:
\begin{equation}\label{eq:munusigma}
\tilde{\mu}(x_j) = 2 \sigma_{j+\frac 12} - \nu_j^n, \qquad \bar{\nu}(x_j) = 2 \sigma_{j-\frac 12} - \mu_j^n.
\end{equation}
Thus, it turns out that the scheme can be rewritten only in terms of the discretized unknowns and of $\sigma_{j\pm \frac 12}$:
\begin{subequations}\label{eq:schema_munu}
  \begin{align}
  & \mu_j^{n+1} = \mu_j^n - c \frac{\Delta t}{\Delta x} (\mu_j^n+\nu_j^n) + \frac{2c\Delta t}{\Delta x} \sigma_{j+\frac 12},   \\
  & \nu_j^{n+1} = \nu_j^n - c \frac{\Delta t}{\Delta x} (\mu_j^n+\nu_j^n) + \frac{2c\Delta t}{\Delta x}  \sigma_{j-\frac 12}.
  \end{align}
\end{subequations}
Or equivalently:
\begin{subequations}\label{eq:schema_rhosigma}
  \begin{align}
    & \rho_j^{n+1} = \rho_j^n - \frac{\Delta t}{\Delta x}(\sigma_{j+\frac 12} - \sigma_{j-\frac 12}),   \label{eq:schema_rho} \\
    & \sigma_j^{n+1} = \sigma_j^n - c\frac{\Delta t}{\Delta x}(2\sigma_j^n - \sigma_{j+\frac 12} - \sigma_{j-\frac 12}).  \label{eq:schema_sigma}
  \end{align}
\end{subequations}

However, solving the stationary systems \eqref{eq:statbar} and \eqref{eq:stattilde} involves the resolution of a nonlinear and nonlocal ODE. Instead, we propose an approximation in the spirit of \cite{Gosse(2016)}.

We replace the nonlinear term in \eqref{eq:barmu}--\eqref{eq:barnu} by $a_{j-\frac 12}^n \cdot \frac{\bar{\nu} - \bar{\mu}}{2c}$, where $a_{j-\frac 12}^n$ stands for a fixed and consistent discretization of $a\left[ \frac{\bar{\nu} - \bar{\mu}}{2c} \right]$ on the interval $(x_{j-1}, x_j)$, to be specified afterwards. Similarly, we will replace  the nonlinear term in \eqref{eq:tildemu}--\eqref{eq:tildenu} by $a_{j+\frac 12}^n \cdot \frac{\tilde{\nu} - \tilde{\mu}}{2c}$ with $a_{j+\frac 12}^n$ defined accordingly. In the following, we detail the construction for the problem \eqref{eq:barmu}--\eqref{eq:barnu} on $(x_{j-1}, x_j)$.

Obviously, the definition of $a_{j-\frac 12}^n$ should be taken with care \cite{Gosse(2016), Carrillo(2016)}. In \cite{Delarue}, the authors showed that, when discretizing the product $a[\rho] \rho$, if $a[\rho]$ and $\rho$ were not evaluated at the same point, then the resulting scheme produces the wrong dynamics. To take this into account, we will split $\rho$ into one contribution coming from the left and one contribution coming from the right, i.e. we set $\bar{\rho}=\rho_L+\rho_R$ and $\bar{\sigma}=\sigma_L+\sigma_R$ where $\rho_L(\Delta x) = 0$ and $\rho_R(0) = 0$. This implies that $\bar{\rho}(\Delta x) = \rho_R(\Delta x)$ and $\bar{\rho}(0) = \rho_L(0)$.

More precisely, we solve the two following boundary value problem, on $(0,\Delta x)$,
\begin{subequations}\label{eq:statLR}
\begin{align}
  & \eps c^2 \frac{d}{dx} {\rho_L} = a_{j-\frac 12,L}^n \rho_L - \sigma_L, \qquad \rho_L(\Delta x) = 0, \label{eq:statL} \\
  & \eps c^2 \frac{d}{dx} {\rho_R} = a_{j-\frac 12,R}^n \rho_R - \sigma_R, \qquad \rho_R(0) = 0, \label{eq:statR}
\end{align}
\end{subequations}
We may solve explicitely these linear systems and, since $\rho_L(0) = \bar{\rho}(0)$ and $\rho_R(\Delta x) = \bar{\rho}(\Delta x)$, obtain the relations
\begin{equation}\label{eq:sigmaLR}
\sigma_L = \bar{\rho}(0) \kappa_{j-\frac 12,L}^n, \qquad \qquad \sigma_R = \bar{\rho}(\Delta x) \kappa_{j-\frac 12,R}^n.
\end{equation}
with
\begin{align}
  & \kappa_{j-\frac 12,L}^n = \frac{a_{j-\frac 12,L}^n}{1-\exp (-a_{j-\frac 12,L}^n \Delta x /(\eps c^2))},  \qquad
  \kappa_{j-\frac 12,R}^n = \frac{a_{j-\frac 12,R}^n}{1-\exp (a_{j-\frac 12,R}^n \Delta x /(\eps c^2))}.  \label{eq:cLcR}
\end{align}
Notice that we have
\begin{equation}\label{eq:limkappa}
\kappa_{j-\frac 12,L}^n \to (a_{j-\frac 12,L}^n)_+, \qquad
\kappa_{j-\frac 12,R}^n \to - (a_{j-\frac 12,R}^n)_-, \quad \text{ when } \eps\to 0,
\end{equation}
where we denote $a_+ = \max(0,a) \geq 0$ and $a_- = \max(0,-a) \geq 0$ the positive and negative negative part of $a$.
Using the boundary conditions in \eqref{eq:statbar}, we have:
\begin{equation}\label{eq:rhobar}
\bar{\rho}(0) = \frac{\nu_{j-1}^n-\bar{\mu}(0)}{2c}, \qquad
\bar{\rho}(\Delta x) = \frac{\bar{\nu}(\Delta x) - \mu_{j}^n}{2c}.
\end{equation}
with \eqref{eq:sigmaLR} and the fact that $\bar{\sigma} = \sigma_L + \sigma_R$ is constant on $[0, \Delta x]$, we get the following $2 \times 2$ system on the unknowns $\bar{\mu}(0), \bar{\nu}(\Delta x)$:
\begin{subequations}
  \begin{align}
    & \mu_j^n + \bar{\nu}(\Delta x) = \bar{\mu}(0) + \nu_{j-1}^n,   \\
    & \mu_j^n + \bar{\nu}(\Delta x) = \frac{\nu_{j-1}^n-\bar{\mu}(0)}{2c} \kappa_{j-\frac 12,L}^n +  \frac{\bar{\nu}(\Delta x) - \mu_{j}^n}{2c} \kappa_{j-\frac 12,R}^n
  \end{align}
\end{subequations}
Solving this system yields:
\begin{subequations}
  \begin{align}
    & \bar{\mu}(0) = - \nu_{j-1}^n \frac{c - \kappa_{j-\frac 12,R}^n - \kappa_{j-\frac 12,L}^n}{c - \kappa_{j-\frac 12,R}^n + \kappa_{j-\frac 12,L}^n} - \mu_j^n \frac{\kappa_{j-\frac 12,R}^n}{c - \kappa_{j-\frac 12,R}^n + \kappa_{j-\frac 12,L}^n},   \\
    & \bar{\nu}(\Delta x) = \nu_{j-1}^n \frac{\kappa_{j-\frac 12,L}^n}{c - \kappa_{j-\frac 12,R}^n + \kappa_{j-\frac 12,L}^n} - \mu_j^n \frac{c + \kappa_{j-\frac 12,R}^n + \kappa_{j-\frac 12,L}^n}{c - \kappa_{j-\frac 12,R}^n + \kappa_{j-\frac 12,L}^n}.
  \end{align}
\end{subequations}
From which we deduce with \eqref{eq:rhobar}
\begin{subequations}\label{eq:rhoRrhoL}
  \begin{align}
   & \rho_{j-\frac 12,L}^n := \bar{\rho}(0) = \frac 1c \left(\frac{(c-\kappa_{j-\frac 12,R}^n)\nu_{j-1}^n+\kappa_{j-\frac 12,R}^n\mu_j^n}{c+\kappa_{j-\frac 12,L}^n-\kappa_{j-\frac 12,R}}\right) \\
   & \rho_{j-\frac 12,R}^n := \bar{\rho}(\Delta x) = \frac 1c \left(\frac{\kappa_{j-\frac 12,L}^n \nu_{j-1}^n-(c+\kappa_{j-\frac 12,L}^n)\mu_j^n}{c+\kappa_{j-\frac 12,L}^n-\kappa_{j-\frac 12,R}}\right)
  \end{align}
\end{subequations}
and with \eqref{eq:sigmaLR}
\begin{equation}\label{eq:sigma12}
  \bar{\sigma}_{j-\frac 12} := \sigma_L+\sigma_R = \rho_{j-\frac 12,L}^n \kappa_{j-\frac 12,L}^n + \rho_{j-\frac 12,R}^n \kappa_{j-\frac 12,R}^n
  = \frac{\nu_{j-1}^n \kappa_{j-\frac 12,L}^n - \mu_{j}^n \kappa_{j-\frac 12,R}^n}{c-\kappa_{j-\frac 12,R}^n+\kappa_{j-\frac 12,L}^n},
\end{equation}
(the above quantities are well-defined since $\kappa_{j-\frac 12,L}^n \geq 0$ and $\kappa_{j-\frac 12,R}^n \leq 0$).
Injecting into \eqref{eq:schema_rhosigma}, it gives the following scheme
\begin{subequations}\label{wb:munu}
\begin{align}
  & \mu_j^{n+1} = \left(1 - \frac{c\Delta t}{\Delta x}\right) \mu_j^n - \frac{c\Delta t}{\Delta x} \frac{c - \kappa_{j+\frac 12, R}^n - \kappa_{j+\frac 12, L}^n}{c - \kappa_{j+\frac 12, R}^n + \kappa_{j+\frac 12, L}^n} \nu_j^n - \frac{2c\Delta t}{\Delta x} \frac{\kappa_{j+\frac 12, R}^n}{c - \kappa_{j+\frac 12, R}^n + \kappa_{j+\frac 12, L}^n} \mu_{j+1}^n,   \label{wb:mu} \\
  & \nu_j^{n+1} = \left(1 - \frac{c\Delta t}{\Delta x}\right) \nu_j^n - \frac{c\Delta t}{\Delta x} \frac{c + \kappa_{j-\frac 12, R}^n + \kappa_{j-\frac 12, L}^n}{c - \kappa_{j-\frac 12, R}^n + \kappa_{j-\frac 12, L}^n} \mu_j^n + \frac{2c\Delta t}{\Delta x} \frac{\kappa_{j-\frac 12, L}^n}{c - \kappa_{j-\frac 12, R}^n + \kappa_{j-\frac 12, L}^n} \nu_{j-1}^n,  \label{wb:nu}
\end{align}
\end{subequations}
where the coefficients $\kappa_{j-\frac 12, L/R}^n$ are defined in \eqref{eq:cLcR}.
Equivalently for the variable $(\rho,\sigma)$ the scheme reads
\begin{subequations}\label{wb:rhosigma}
\begin{align}
  & \rho_j^{n+1} = \rho_j^n - \frac{\Delta t}{\Delta x} \left( \frac{\nu_j^n \kappa_{j+\frac 12, L}^n - \mu_{j+1}^n \kappa_{j+\frac 12, R}^n}{c-\kappa_{j+\frac 12, R}^n+\kappa_{j+\frac 12, L}^n} - \frac{\nu_{j-1}^n \kappa_{j-\frac 12, L}^n - \mu_{j}^n \kappa_{j-\frac 12, R}^n}{c-\kappa_{j-\frac 12, R}^n+\kappa_{j-\frac 12, L}^n} \right)  \label{wb:rho} \\
  & \sigma_j^{n+1} = \sigma_j^n - c \frac{\Delta t}{\Delta x} \left( 2\sigma_j^n - \frac{\nu_j^n \kappa_{j+\frac 12, L}^n - \mu_{j+1}^n \kappa_{j+\frac 12, R}^n}{c-\kappa_{j+\frac 12, R}^n+\kappa_{j+\frac 12, L}^n} - \frac{\nu_{j-1}^n \kappa_{j-\frac 12, L}^n - \mu_{j}^n \kappa_{j-\frac 12, R}^n}{c-\kappa_{j-\frac 12, R}^n+\kappa_{j-\frac 12, L}^n} \right),
\end{align}
\end{subequations}
where we recall that $\mu_j^n = \sigma_j^n - c\rho_j^n$ and $\nu_j^n = \sigma_j^n + c\rho_j^n$.

It remains to define the velocities $a_{j-\frac 12, L/R}^n$ used in \eqref{eq:statLR} and in \eqref{eq:cLcR}.
We take
$$
a_{j-\frac 12,L/R}^n = - \sum_{k \neq j} W'(x_j - x_k) \rho_{k-\frac 12,L/R}^n.
$$
However, this discretization implies the resolution of a nonlinear problem, since the quantities $\rho_{k-\frac 12,L/R}^n$ depends nonlinearly on $a_{j-\frac 12,L/R}^n$.


Then, we implement a fixed point method initialized with $a_{j-\frac 12,L}^{n, (0)} := a_{j-1}^n$ and $a_{j-\frac 12,R}^{n, (0)} := a_j^n$. Solving, on each cell $(x_{j-1}, x_j)$, the system of ODEs \eqref{eq:statLR} with these values for the velocities gives two sequences $(\rho_{j-\frac 12,L}^{(1)})_{j \in \ZZ}$ and $(\rho_{j-\frac 12,R}^{(1)})_{j \in \ZZ}$. Then, we assign the next value of the velocity to $a_{j-\frac 12,L/R}^{n, (1)} := - \ds \sum_{k \neq j} W'(x_j - x_k) \rho_{k-\frac 12,L/R}^{(1)}$, which allows us to compute new values for the left and right densities $(\rho_{j-\frac 12,L}^{(2)})_{j \in \ZZ}$ and $(\rho_{j-\frac 12,R}^{(2)})_{j \in \ZZ}$ through \eqref{eq:statLR}. We iterate until $W_2(\rho_L^{(i)}, \rho_L^{(i+1)})$ and $W_2(\rho_R^{(i)}, \rho_R^{(i+1)})$ pass below a certain threshold. Notice that the velocities $a_{j-\frac 12,L/R}^{n, (i)}$ always remain bounded by $a_\infty$. In practice, only a few iterations are needed.

The resulting scheme is consistent for any $\eps > 0$ and stable under standard stability conditions, as show the following lemmas.
\begin{lemma}[$L^1$ stability]\label{lemma:stabL1}
Under the CFL condition $\frac{c\Delta t}{\Delta x} \leq 1$ and the subcharacteristic condition $c \geq a_\infty$, there holds that the sequence $(\mu_j^n,\nu_j^n)_{j,n}$ defined by the scheme \eqref{wb:munu} verifies the following $L^1$ stability property:
\begin{align*}
\forall n \in \NN, \qquad \sum_{j \in \ZZ} \left(|\mu_j^{n+1}| + |\nu_j^{n+1}|\right) \leq \sum_{j \in \ZZ} \left(|\mu_j^n| + |\nu_j^n|\right).
\end{align*}
\end{lemma}
\begin{proof}
In each combination of \eqref{wb:munu}, the first coefficient is nonnegative under the CFL condition $\frac{c\Delta t}{\Delta x} \leq 1$, and so is the last one since $\kappa_{j\pm \frac 12,L}^n \geq 0$ and $\kappa_{j\pm \frac 12,R}^n \leq 0$. Moreover, under the subcharacteristic condition $c \geq a_\infty$, it holds that $-c \leq \kappa_{j\pm \frac 12,R} + \kappa_{j\pm \frac 12,R} \leq c$ so the remaining coefficient is nonpositive. Thus, applying the triangle inequality and reindexing the sums appropriately,
\begin{align*}
\sum_{j \in \ZZ} \left(|\mu_j^{n+1}| + |\nu_j^{n+1}|\right) & \leq \sum_{j \in \ZZ} \left(1 - \frac{c\Delta t}{\Delta x}\right) |\mu_j^n| + \sum_{j \in \ZZ} \frac{c\Delta t}{\Delta x} \frac{c - \kappa_{j+\frac 12, R}^n - \kappa_{j+\frac 12, L}^n}{c - \kappa_{j+\frac 12, R}^n + \kappa_{j+\frac 12, L}^n} |\nu_j^n|  \\
                 &\quad - \sum_{j \in \ZZ} \frac{2c\Delta t}{\Delta x} \frac{\kappa_{j+\frac 12, R}^n}{c - \kappa_{j+\frac 12, R}^n + \kappa_{j+\frac 12, L}^n} |\mu_{j+1}^n| + \sum_{j \in \ZZ} \left(1 - \frac{c\Delta t}{\Delta x}\right) |\nu_j^n| \\
  & \quad + \sum_{j \in \ZZ} \frac{c\Delta t}{\Delta x} \frac{c + \kappa_{j+\frac 12, R}^n + \kappa_{j+\frac 12, L}^n}{c - \kappa_{j+\frac 12, R}^n + \kappa_{j+\frac 12, L}^n} |\mu_{j+1}^n| + \frac{2c\Delta t}{\Delta x} \frac{\kappa_{j+\frac 12, L}^n}{c - \kappa_{j+\frac 12, R}^n + \kappa_{j+\frac 12, L}^n} |\nu_j^n|, \\
& \leq \left(1 - \frac{c\Delta t}{\Delta x}\right) \sum_{j \in \ZZ} \left(|\mu_j^n| + |\nu_j^n|\right) + \frac{c\Delta t}{\Delta x} \sum_{j \in \ZZ} |\mu_{j+1}^n| + \frac{c\Delta t}{\Delta x} \sum_{j \in \ZZ} |\nu_j^n|, \\
& \leq \sum_{j \in \ZZ} \left(|\mu_j^n| + |\nu_j^n|\right).
\end{align*}
It concludes the proof.
\end{proof}
\begin{lemma}[Consistency for smooth solutions]\label{consistance}
Assume that, for all $j \in \ZZ$, we have $a_{j-\frac 12, L/R}^n = - \ds \sum_{k \neq j} W'(x_j - x_k) \rho_{k-\frac 12, L/R}$. Then, for any $\eps > 0$, the scheme \eqref{eq:schema_rhosigma} is consistent with \eqref{eq:agg_relax} provided that the solutions are smooth enough.
\end{lemma}
\begin{proof}
For $j \in \ZZ$, one has, using Taylor expansions as $\Delta x \to 0$,
\begin{align*}
& \frac{\kappa_{j-\frac 12, L}^n}{c - \kappa_{j-\frac 12, R}^n + \kappa_{j-\frac 12, L}^n} = \frac 12 - \frac{1}{4\eps c^2} \left(c - \frac{a_{j-\frac 12, L}^n + a_{j-\frac 12, R}^n}{2}\right) \Delta x + O(\Delta x^2), \\
& \frac{\kappa_{j-\frac 12, R}^n}{c - \kappa_{j-\frac 12, R}^n + \kappa_{j-\frac 12, L}^n} = - \frac 12 + \frac{1}{4\eps c^2} \left(c + \frac{a_{j-\frac 12, L}^n + a_{j-\frac 12, R}^n}{2}\right) \Delta x + O(\Delta x^2).
\end{align*}
Thus,
\begin{align*}
\sigma_{j-\frac 12} = \frac{\sigma_{j-1}^n + \sigma_j^n}{2} + c \frac{\rho_{j-1}^n - \rho_j^n}{2} & - \frac{1}{4\eps c^2} \Bigg( \left(c - \frac{a_{j-\frac 12, L}^n + a_{j-\frac 12, R}^n}{2}\right) (\sigma_{j-1}^n + c \rho_{j-1}^n) \\
& + \left(c + \frac{a_{j-\frac 12, L}^n + a_{j-\frac 12, R}^n}{2}\right) (\sigma_j^n - c\rho_j^n) \Bigg) \Delta x + O(\Delta x^2).
\end{align*}
In particular, $\sigma_{j-\frac 12}$ is clearly consistent with $\sigma(t^n, x_{j-\frac 12})$ as long as the solution $(\rho, \sigma)$ is smooth enough to perform standard consistency analysis for finite differences. This shows that \eqref{eq:schema_rho} is consistent with $\pa_t \rho + \pa_x \sigma = 0$. As for the consistency of \eqref{eq:schema_sigma} with $\pa_t \sigma + c^2 \pa_x \rho = \frac{1}{\eps} (a[\rho] \rho - \sigma)$, we write:
\begin{align*}
& \sigma_{j+\frac 12} + \sigma_{j-\frac 12} - 2\sigma_j^n = \frac{\sigma_{j+1}^n - 2\sigma_j^n + \sigma_{j-1}^n}{2} + c \frac{\rho_{j-1}^n - \rho_{j+1}^n}{2} - \frac{ \Delta x}{4\eps c^2} \Bigg[ c (\sigma_{j-1}^n + 2\sigma_j^n + \sigma_{j+1}^n) \\
&\quad + \frac{a_{j-\frac 12, L}^n + a_{j-\frac 12, R}^n}{2} (\sigma_j^n - \sigma_{j-1}^n) + \frac{a_{j+\frac 12, L}^n + a_{j+\frac 12, R}^n}{2} (\sigma_{j+1}^n - \sigma_j^n) + c^2 (\rho_{j-1}^n - \rho_{j+1}^n) \\
&\quad  - c \Bigg( \frac{a_{j-\frac 12, L}^n + a_{j-\frac 12, R}^n}{2} \rho_{j-1}^n
 + \frac{a_{j-\frac 12, L}^n + a_{j-\frac 12, R}^n + a_{j+\frac 12, L}^n + a_{j+\frac 12, R}^n}{2} \rho_j^n + \frac{a_{j+\frac 12, L}^n + a_{j+\frac 12, R}^n}{2} \rho_{j+1}^n \Bigg) \Bigg] + O(\Delta x^2).
\end{align*}
Using Taylor expansions, we have, for smooth solutions $\sigma(t^n, x_{j+1}) - 2\sigma(t^n, x_j) + \sigma(t^n, x_{j-1}) = O(\Delta x^2)$, $\rho(t^n, x_{j-1}) - \rho(t^n, x_{j+1}) = O(\Delta x)$, $\sigma(t^n, x_j) - \sigma(t^n, x_{j-1}) = O(\Delta x)$ and $\sigma(t^n, x_{j+1}) - \sigma(t^n, x_j) = O(\Delta x)$. Along with the bound $|a_{j\pm\frac 12, L/R}^n| \leq a_\infty$, this implies:
\begin{align*}
\sigma_{j+\frac 12} + \sigma_{j-\frac 12} - 2\sigma_j^n & = c \frac{\rho_{j-1}^n - \rho_{j+1}^n}{2} - \frac{1}{4\eps c^2} \Bigg[ c (\sigma_{j-1}^n + 2\sigma_j^n + \sigma_{j+1}^n)  - c \Bigg( \frac{a_{j-\frac 12, L}^n + a_{j-\frac 12, R}^n}{2} \rho_{j-1}^n \\
& + \frac{a_{j-\frac 12, L}^n + a_{j-\frac 12, R}^n + a_{j+\frac 12, L}^n + a_{j+\frac 12, R}^n}{2} \rho_j^n + \frac{a_{j+\frac 12, L}^n + a_{j+\frac 12, R}^n}{2} \rho_{j+1}^n \Bigg) \Bigg] \Delta x + O(\Delta x^2).
\end{align*}
Clearly, $c \frac{\rho_{j-1}^n - \rho_{j+1}^n}{2}$ and $c (\sigma_{j-1}^n + 2\sigma_j^n + \sigma_{j+1}^n)$ are consistent with accuracy $O(\Delta x^2)$ and $O(\Delta x)$ respectively with $-c \pa_x \rho(t^n, x_j)$ and $4c \sigma(t^n, x_j)$. For the remaining terms, let us recall that, with the notations of \eqref{eq:rhobar}:
$$
\rho_{j-\frac 12, L} = \frac{\nu_{j-1}^n-\bar{\mu}(0)}{2c} = \frac{\nu_{j-1}^n - \sigma_{j-\frac 12}}{c}, \qquad
\rho_{j-\frac 12, R} = \frac{\bar{\nu}(\Delta x) - \mu_{j}^n}{2c} = \frac{\sigma_{j-\frac 12} - \mu_j}{c}.
$$
Hence $\rho_{j-\frac 12, L} + \rho_{j-\frac 12, R} = \dfrac{\nu_{j-1}^n - \mu_j^n}{c} = \dfrac{\sigma_{j-1}^n - \sigma_j^n}{c} + \rho_{j-1}^n + \rho_j^n$. Since $\sigma(t^n, x_{j-1}) - \sigma(t^n, x_j) = O(\Delta x)$, and assuming that:
$$
a_{j-\frac 12, L/R}^n = - \ds \sum_{k \neq j} W'(x_j - x_k) \rho_{k-\frac 12, L/R}
$$
we deduce that $a_{j-\frac 12, L}^n + a_{j-\frac 12, R}^n $ is consistent with $a[\rho(t^n)](x_{j-1}) + a[\rho(t^n)](x_j)$ with accuracy $O(\Delta x)$. It follows that $\sigma_{j+\frac 12} + \sigma_{j-\frac 12} - 2\sigma_j^n$ is consistent with $-\pa_x \rho(t^n, x_j) - \frac{1}{\eps} \Big(\sigma(t^n, x_j) - a[\rho(t^n)](x_j) \rho(t^n, x_j)\Big)$, again with accuracy $O(\Delta x)$, and this concludes the proof.
\end{proof}
The stability conditions in Lemma \ref{lemma:stabL1} are independent on $\eps$, we recover in the limit $\eps \to 0$, using \eqref{eq:limkappa}, the scheme of \cite{Gosse(2016)}:
\begin{subequations}\label{dis:GV}
\begin{align}
  & \rho_j^{n+1} = \rho_j^n - \frac{\Delta t}{\Delta x} \left( \frac{\nu_j^n (a_{j+\frac 12, L}^n)_+ + \mu_{j+1}^n (a_{j+\frac 12, R}^n)_-}{c+(a_{j+\frac 12, R}^n)_- +(a_{j+\frac 12, L}^n)_+} - \frac{\nu_{j-1}^n (a_{j-\frac 12, L}^n)_+ + \mu_{j}^n (a_{j-\frac 12, R}^n)_-}{c + (a_{j-\frac 12, R}^n)_- + (a_{j-\frac 12, L}^n)_+} \right) \label{dis:GVrho}  \\
  & \sigma_j^{n+1} = \sigma_j^n - c \frac{\Delta t}{\Delta x} \left( 2\sigma_j^n - \frac{\nu_j^n (a_{j+\frac 12, L}^n)_+ + \mu_{j+1}^n (a_{j+\frac 12, R}^n)_-}{c+(a_{j+\frac 12, R}^n)_- +(a_{j+\frac 12, L}^n)_+} - \frac{\nu_{j-1}^n (a_{j-\frac 12, L}^n)_+ + \mu_{j}^n (a_{j-\frac 12, R}^n)_-}{c + (a_{j-\frac 12, R}^n)_- + (a_{j-\frac 12, L}^n)_+} \right), \label{dis:GVsigma}
\end{align}
\end{subequations}
which is stable under the conditions $\frac{c \Delta t}{\Delta x} \leq 1$ and $c \geq a_\infty$.
Notice that with the notation in \eqref{eq:sigma12}, equation \eqref{dis:GVrho} may be rewritten
$$
\rho_j^{n+1} = \rho_j^n - \frac{\Delta t}{\Delta x} \left( \rho_{j+\frac 12,L}^n (a_{j+\frac 12,L}^n)_+ - \rho_{j+\frac 12,R}^n (a_{j+\frac 12,R}^n)_- -  \rho_{j-\frac 12,L}^n (a_{j-\frac 12,L}^n)^+ + \rho_{j-\frac 12,R}^n (a_{j-\frac 12,R}^n)_-\right).
$$

\section{Numerical experiments}\label{sec:numresult}

We present some numerical illustrations for the two schemes described in the previous section. In addition to the potential $W(x) = \frac{|x|}{2}$, we also consider the smooth potential $W(x) = \frac{x^2}{2}$.

Numerical tests are conducted on the domain $[-1,1]$ with the inital data $\rho_0 = \frac{1}{2} \delta_{-0.5} + \frac{1}{2} \delta_{0.5}$, $\sigma_0 = a[\rho_0] \rho_0$ and both schemes are initialized with
$$
\rho_j^0 = \dfrac{1}{\Delta x} \rho_0(C_j), \qquad \sigma_j^0 = \dfrac{1}{\Delta x} \sigma_0(C_j).
$$
Figure \ref{fig:1} shows that both schemes recover the correct dynamics in the limit $\eps \to 0$: for the potential $W(x) = \frac{|x|}{2}$, one can compute the exact velocity of both Dirac masses for the aggregation equation \eqref{eq:aggreg} and see that they should be located respectively in $x = -0.2$ and $x = 0.2$ in final time $T = 1.2$.

This test is set up with $\eps = 10^{-7}$, on a cartesian mesh of $[-1, 1]$ with 1500 cells, $c = 1$ and the CFL $c \dfrac{\Delta t}{\Delta x} = 0.9$. Both schemes \eqref{dis:rusanov} and \eqref{dis:GV} display the correct velocity for the Dirac masses, but one can notice that the Rusanov scheme \eqref{dis:rusanov} shows more numerical diffusion. Note that both schemes being written in conservation form, they preserve the total mass of $\rho$, which is also verified numerically.
\begin{figure}[h]
\centering
\includegraphics[scale = 0.5]{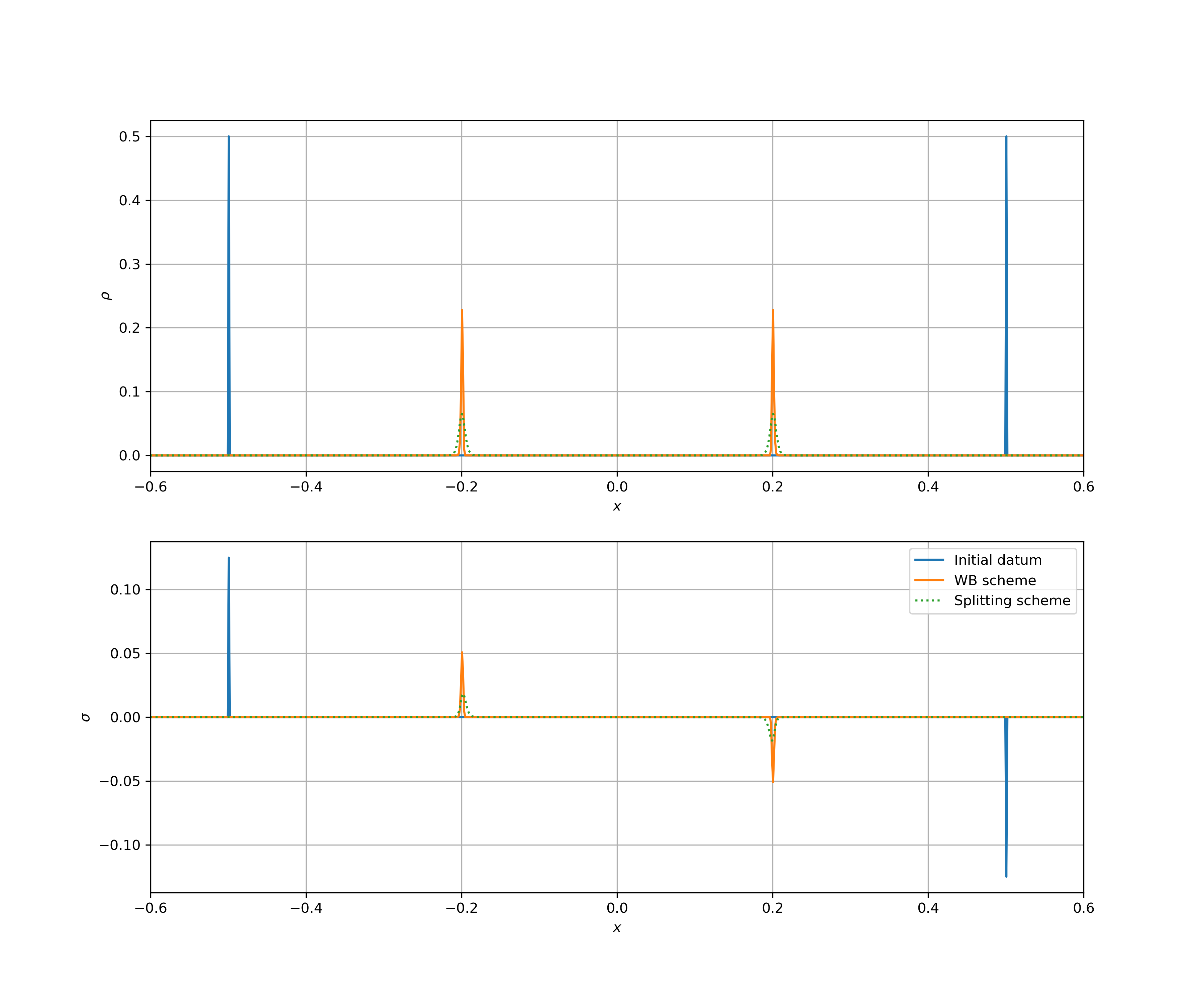}
\caption{Dynamics of two Dirac masses for the potential $W(x) = \frac{|x|}{2}$ in time $T = 1.2$.}
\label{fig:1}
\end{figure}

We then investigate the order of convergence when $\Delta x$ goes to 0 with $\eps$ fixed, in Wasserstein distance $W_1$ (the numerical results are the same for $W_2$).

After performing tests for several values of $\eps$, it appears that the convergence rate does not depend on the size of $\eps$. Therefore, as an example, we propose simulations in final time $T = 0.5$, with the same intial data and stability parameters as above, and with $\eps = 2 \times 10^{-6}$ for Figure \ref{fig:2} and with $\eps = 10^{-2}$ for Figure \ref{fig:3}.

For a fixed value of $\eps$, both schemes seem to converge with order $1/2$ with respect to $\Delta x$ for the smooth potential $W(x) = \frac{x^2}{2}$ (see Figure \ref{fig:2}) whereas they seem of order 1 for the potential $W(x) = \frac{|x|}{2}$ (see Figure \ref{fig:3}). This can be explained as both schemes possess some numerical diffusion which is somehow counterbalanced by the aggregation phenomenon in the case of a pointy potential, as already observed in \cite{Fabreges(2019)}. Due to the link with the Burgers equation, this superconvergence phenomenon is directly linked to the results of Després \cite{Despres} which should be rigorously extended to our case (the mere extension to the upwind scheme of \cite{Carrillo(2016)} for the aggregation is not straightforward).
\begin{figure}[h]
\centering
\includegraphics[scale = 0.5]{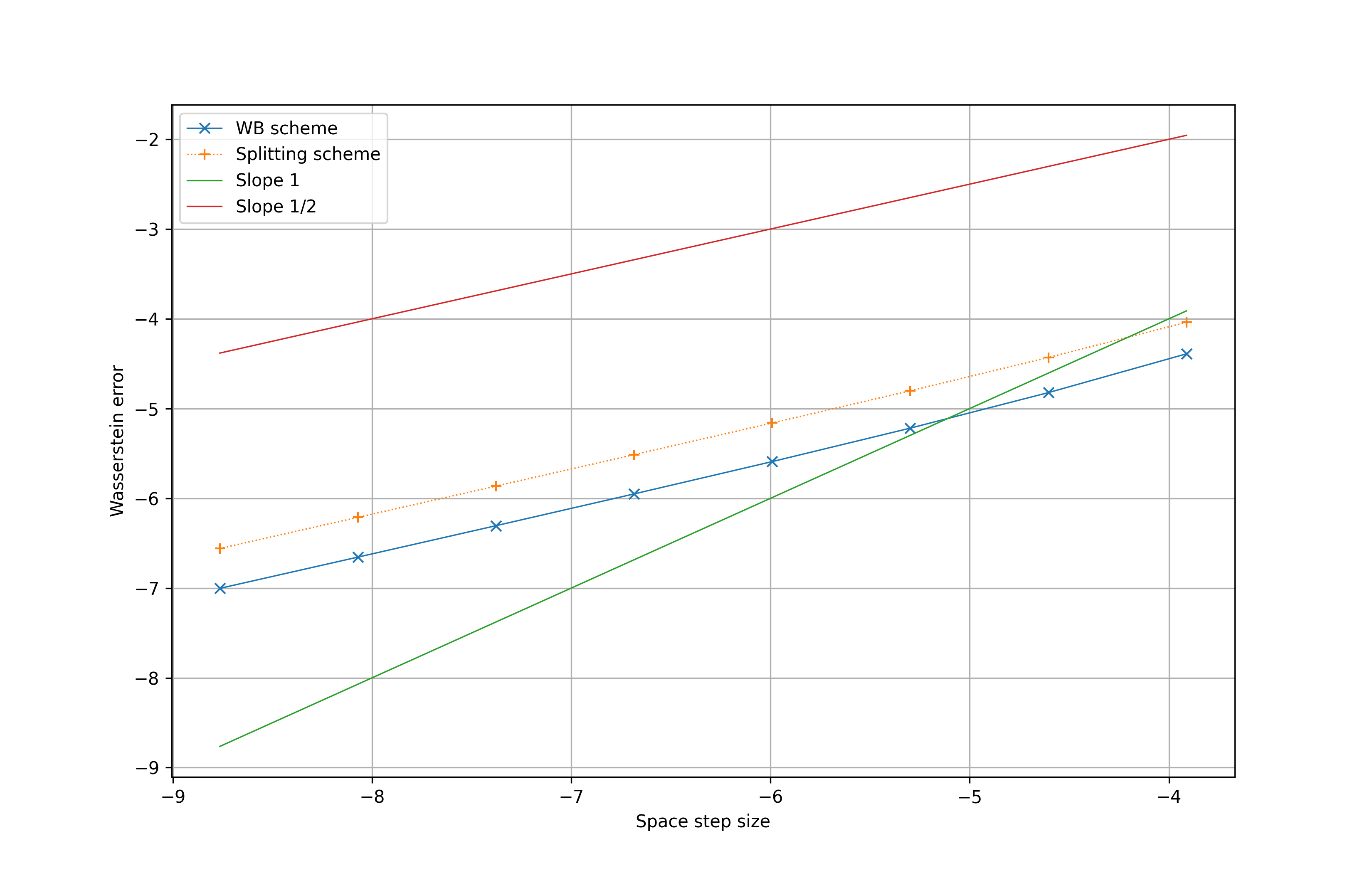}
\caption{Order of convergence of the splitting scheme and the well-balanced scheme for the smooth potential $W(x) = \frac{x^2}{2}$.}
\label{fig:2}
\end{figure}
\begin{figure}[h]
\centering
\includegraphics[scale = 0.5]{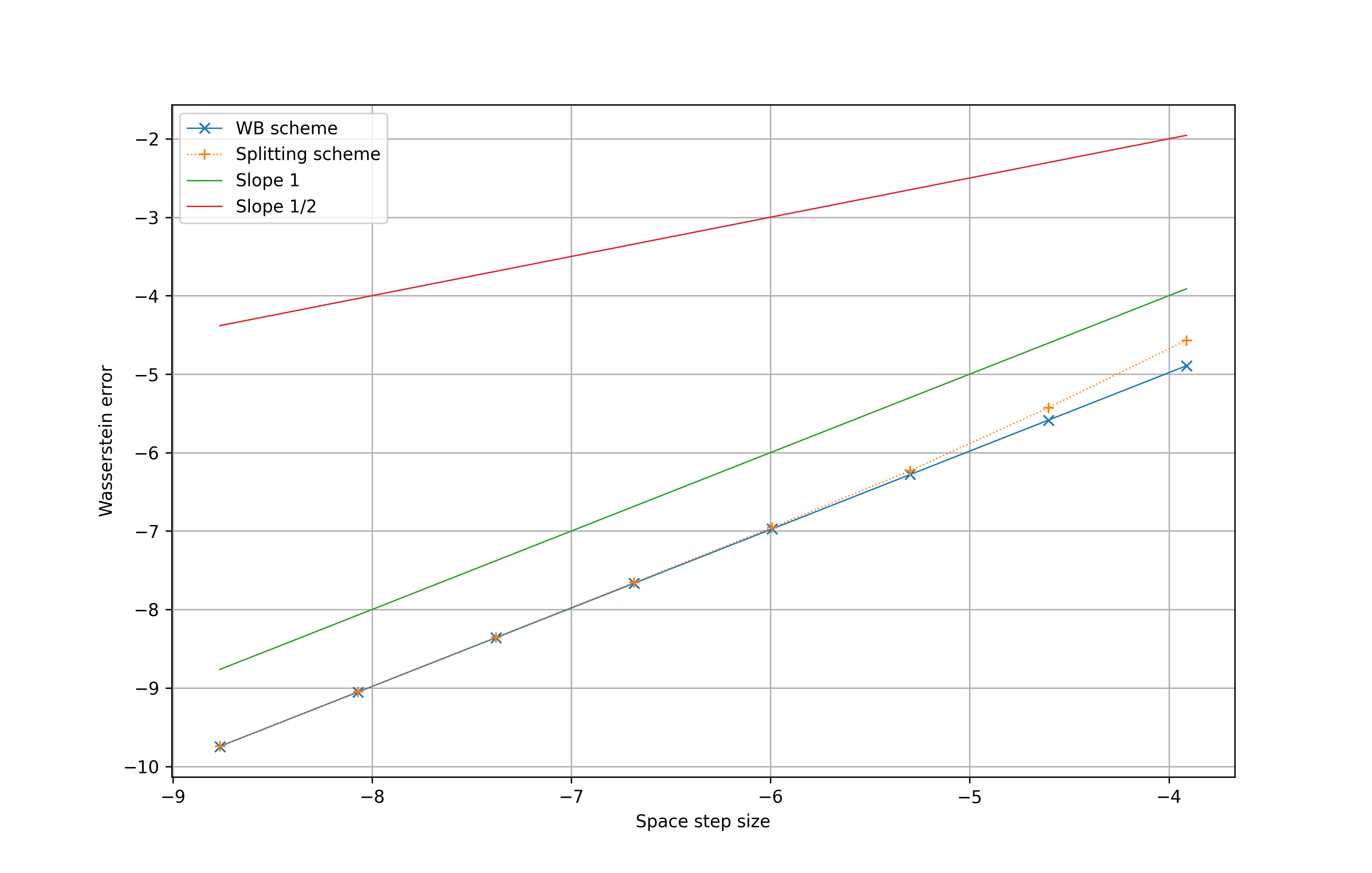}
\caption{Order of convergence of the splitting scheme and the well-balanced scheme for the pointy potential $W(x) = \frac{|x|}{2}$.}
\label{fig:3}
\end{figure}
Finally, we also verify the well-balanced property of the scheme
\eqref{wb:rhosigma} by computing the $W_1$ distance between the
approximated solution at time $T=0.5$ and the stationary solution
of \eqref{eq:agg_relax} given by:
$$
\rho(t,x) = \rho_0(x) := \frac{1}{8\eps c^2}\Bigg(1 - \tanh^2\Big(\frac{x}{4\eps c^2}\Big)\Bigg).
$$
The test is conducted with $\eps = 2 \times 10^{-4}$, with the exact boundary conditions given by the above formula, and for several values of $\Delta x$.
As we show in Figure \ref{fig:4}, the scheme
\eqref{wb:rhosigma} preserves well the above equilibrium for any
$\Delta x$ (although we have replaced the resolution of the systems \eqref{eq:statbar} and \eqref{eq:stattilde} with linear systems, see \eqref{eq:statLR}),
while, for the splitting scheme, we recover the linear
convergence towards $\rho_0$ which is, in this case, the exact solution.
\begin{figure}[h]
\centering
\includegraphics[scale = 0.75]{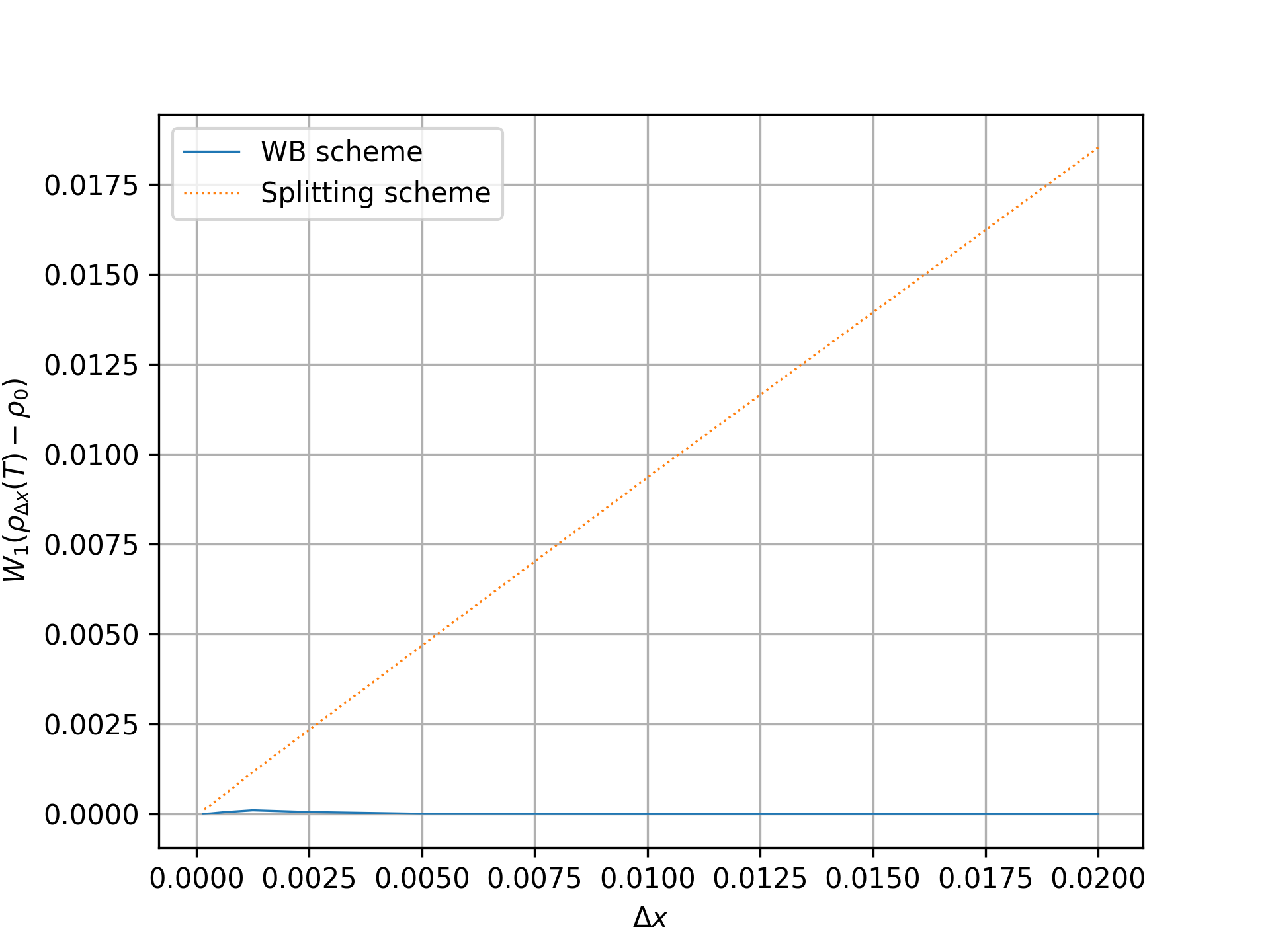}
\caption{Distance to the equilibrium for the splitting scheme and the well-balanced scheme and for the pointy potential $W(x) = \frac{|x|}{2}$.}
\label{fig:4}
\end{figure}

  \bibliographystyle{plain}
\bibliography{Relaxation_limit.bib}


\end{document}